\documentclass[a4paper,dvipsnames]{article}

\usepackage{amsthm,amssymb,amsmath,enumerate,graphicx,epsf,bm,caption,framed}
\usepackage{authblk}
\usepackage[greek,english]{babel} 
\usepackage[percent]{overpic}
\usepackage{epstopdf}

\newcommand{\COLORON}{0}
\newcommand{\NOTESON}{0}
\newcommand{\Debug}{0}
\usepackage[usenames]{color} 
\usepackage{amsthm,amssymb,amsmath,bbm,enumerate,graphicx,epsf,stmaryrd}
\usepackage[bookmarks, colorlinks=false, breaklinks=true]{hyperref} 

\newcommand{\comment}[1]{}
\newcommand{\COMMENT}[1]{}

\definecolor{darkgray}{rgb}{0.3,0.3,0.3}
\newcommand{\defi}[1]{{\color{darkgray}\emph{#1}}}

\comment{
	\begin{lemma}\label{}	
\end{lemma}
\begin{proof}

\end{proof}

\begin{theorem}\label{}
\end{theorem} 
\begin{proof} 	

\end{proof}

}

\newtheorem{proposition}{Proposition}[section]

\newtheorem{theorem}[proposition]{Theorem}
\newtheorem{corollary}[proposition]{Corollary}

\newtheorem{lemma}[proposition]{Lemma}

\newtheorem{examp}[proposition]{Example}

\newcommand{\FIG}{0}

\ifnum \NOTESON = 1 \newcommand{\note}[1]{ 

\hspace*{-30pt}
	{\color{blue}  NOTE: \color{Turquoise}{\small  \tt \begin{minipage}[c]{1.1\textwidth}  #1 \end{minipage} \ignorespacesafterend }} 
	
	}
\else \newcommand{\note}[1]{} \fi

\newcommand{\afsubm}[1]{ \ifnum \Debug = 1 {\mymargin{#1}}
\fi} 

\ifnum \Debug = 1 
\else  \fi

\ifnum \FIG = 1 
\else  \fi

\ifnum \FIG = 1 
\else  \fi

\ifnum \Debug = 1 \usepackage[notref,notcite]{showkeys}
\fi

\ifnum \COLORON = 0 \renewcommand{\color}[1]{}
\fi

\newcommand{\N}{\ensuremath{\mathbb N}}
\newcommand{\R}{\ensuremath{\mathbb R}}

\newcommand{\Z}{\ensuremath{\mathbb Z}}

\makeatletter
\DeclareRobustCommand{\cev}[1]{%
  \mathpalette\do@cev{#1}%
}
\newcommand{\do@cev}[2]{%
  \fix@cev{#1}{+}%
  \reflectbox{$\m@th#1\vec{\reflectbox{$\fix@cev{#1}{-}\m@th#1#2\fix@cev{#1}{+}$}}$}%
  \fix@cev{#1}{-}%
}
\newcommand{\fix@cev}[2]{%
  \ifx#1\displaystyle
    \mkern#23mu
  \else
    \ifx#1\textstyle
      \mkern#23mu
    \else
      \ifx#1\scriptstyle
        \mkern#22mu
      \else
        \mkern#22mu
      \fi
    \fi
  \fi
}

\makeatother

\newcommand{\g}{\ensuremath{G\ }}

\renewcommand{\Pr}{\mathbb{P}}

\newcommand{\Lr}[1]{Lemma~\ref{#1}}

\newcommand{\Tr}[1]{Theorem~\ref{#1}}
\newcommand{\Trs}[1]{Theorems~\ref{#1}}
\newcommand{\Sr}[1]{Section~\ref{#1}}

\newcommand{\Prr}[1]{Pro\-position~\ref{#1}}

\newcommand{\labtequ}[2]{
 \begin{equation} \label{#1} 	\begin{minipage}[c]{0.9\textwidth}  #2 \end{minipage} \ignorespacesafterend \end{equation} }

\newcommand{\mymargin}[1]{
 \ifnum \Debug = 1
  \marginpar{%
    \begin{minipage}{\marginparwidth}\small%
      \begin{flushleft}%
        {\color{blue}#1}%
      \end{flushleft}%
   \end{minipage}%
  }%
 \fi
}%

\newcommand{\mySection}[2]{}

\usepackage{xcolor}

\usepackage{mathtools}

\RequirePackage{latexsym} \RequirePackage{amsthm}
\RequirePackage{amsmath}
\usepackage{hyperref}

\usepackage{enumerate}
\RequirePackage{amssymb}\RequirePackage{makeidx}
\usepackage{dsfont}
\usepackage{mathrsfs}

\newcommand{\myremark}[1]{\ifnum \Debug = 1 \tiny #1 \fi}

\newcommand{\pint}{interface}

\newcommand{\ar}[1]{\vec{#1}}

\newcommand{\pcs}{\ensuremath{\dot{p}_c}}

\usepackage{authblk}

\title{On the exponential growth rates of lattice animals and interfaces II:  new asymptotic bounds}

\author[1]{Agelos Georgakopoulos}
\author[2]{Christoforos Panagiotis}
\affil[1,2]{{Mathematics Institute}\\
{University of Warwick}\\
{CV4 7AL, UK}\thanks{Supported by the European Research Council (ERC) under the European Union's Horizon 2020 research and innovation programme (grant agreement No 639046).}\\}
\affil[2]{{Universit{\'e} de Gen{\`e}ve}\\
{Section de Math{\'e}matiques}\\
{rue du Conseil-G{\'e}n{\'e}ral 7-9}\\
{1205 Geneva, Switzerland}}

\begin{document}
\date{}
\maketitle

\begin{abstract}
We introduce a method for translating any upper bound on the percolation threshold  of a lattice \g into a lower bound on the exponential growth rate $a(G)$ of lattice animals  and vice-versa. We exploit this in both directions.  We improve on the best known asymptotic lower and upper bounds on $a(\Z^d)$ as $d\to \infty$. We use percolation as a tool to obtain the latter, and conversely we use the former to obtain lower bounds on $p_c(\Z^d)$. We obtain the rigorous lower bound $\dot{p}_c(\Z^3)>0.2522$ for 3-dimensional site percolation.
\end{abstract}

\section{Introduction}

A \defi{lattice animal} is a connected subgraph $S$ of the hypercubic lattice $\Z^d$. If $S$ is an induced subgraph, which means that it contains every edge of $\Z^d$ with both end-vertices in $S$, then it is called a \defi{lattice site-animal} or \defi{polycube}. Alternatively, a polycube can be defined as a connected set of cubical cells in $\Z^d$. The counts of lattice (site-)animals of size $n$, and their asymptotics as $n$ and $d$ goes to infinity, have been extensively studied by scholars in statistical mechanics as well as  combinatorics and computer science \cite{SiteAnimals,BareShal,DelTai,GP,HammondExpRates,Harr82, PG95,LatticeTrees,LatticeTreesTerms}, both in $\Z^d$ and  other lattices \cite{BaRoShaImp,BaShaZheImp,RanWelAni}. A lot of the motivation comes from the study of random configurations in $\Z^d$, a central theme in many models of statistical mechanics. 

The exact count $a_n(\Z^d)$ of $d$-dimensional lattice animals of size $n$ containing the origin is very difficult to come by even in 2 dimensions, and so the mainstream focuses on their exponential growth rates $a(\Z^d):=\lim_{n\to\infty} {a_n(\Z^d)}^{1/n}$. These have important interactions with statistical mechanics models such as percolation theory, the present paper being an instance of this interaction. Some precise asymptotic expansions for $a(\Z^d)$ and its site-counterpart $\dot{a}(\Z^d)$ were reported in the physics literature \cite{GP,Harr82, PG95} but without any rigorous bounds on the error terms. Miranda and Slade \cite{LatticeTrees,LatticeTreesTerms} determined the first three terms of the $1/d$ expansion of $a(\Z^d)$ rigorously. 

Much less is known about  $\dot{a}(\Z^d)$.
Barequet, Barequet and Rote in \cite{SiteAnimals}  proved that $\dot{a}(\Z^d)=2de-o(d)$. Peard and Gaunt had previously made involved, but nonrigorous, calculations that yield $\dot{a}(\Z^d)=2de-3e+O(1/d)$ \cite[(2.22)]{PG95}, and \cite{SiteAnimals} expressed the belief that this is correct. 
Our first result is that this prediction is indeed right as a lower bound (\Tr{gap}). We deduce this from a recent bounds of Heydenreich and Matzke \cite{HeyMat} on the site percolation threshold 
$\pcs(\Z^d)$, obtained using an involved technique called lace expansion. (The dot in  $\pcs(\Z^d),\dot{a}(\Z^d)$ etc.\ means that we are considering site percolation, or lattice site-animals; most of our results have a bond and a site version.) To do so, we exploit the following formula  that allows us to translate any upper bound on the percolation threshold of a `lattice' \g into a lower bound on the exponential growth rate $a(G)$ of lattice animals (and other creatures) and vice-versa: 
\labtequ{a ge frp}{$\dot{a}(G) \geq f(r(\dot{p}_c(G)))$,}
where $f(r):= \frac{(1+r)^{1+r}}{r^r}$ and $r(p):=\frac{1-p}{p}$ are universal  functions. This formula is proved and discussed in the companion paper \cite{brI}.

The aforementioned upper bound of \cite{SiteAnimals} was improved to $\dot{a}(\Z^d)\leq 2de-2e+1/(2d-2)$ in simultaneous work by Barequet and Shalah \cite{BareShal}. We improve this further asymptotically to $\dot{a}(\Z^d)\leq 2de-5e/2+O(1/\log(d))$ (\Tr{improved bounds}). For this we use direct combinatorial arguments that do not involve percolation. We can then plug these bounds into \eqref{a ge frp} to obtain the bounds\\ $\pcs(\Z^d)\geq \dfrac{1}{2d}+\dfrac{2}{(2d)^2}-O(1/d^2\log(d))$ (\Tr{thm pc lower}). This bound was improved by Heydenreich and Matzke \cite{HeyMat}  shortly after the first draft of our work appeared, see \eqref{site perco asym rig}.

\medskip
Certain sub-families of lattice (site-)animals are of interest as well. The \defi{lattice trees} in particular, i.e.\ the subtrees of $\Z^d$,  have been studied  \cite{AlekBar, LatticeTrees,LatticeTreesTerms} and the first three terms of the $1/d$-expansion of their exponential growth rates  $t(\Z^d)$ are known \cite{LatticeTreesTerms}.
We are interested in an intermediate sub-species, called (lattice) \defi{interfaces}, a family of lattice (site-)animals containing the lattice trees. We introduced our notion of interfaces in \cite{analyticity}, where they played a central role in proving many results about Bernoulli percolation. In the companion paper \cite{brI} we focus on their exponential growth rates $b(\Z^d)$ and $\dot{b}(\Z^d)$, and this paper continues this study: we determine the first  terms of their $1/d$-expansion (\Trs{asymptotics} and~\ref{a Zd}).




\medskip
In this paper we used percolation as a tool to bound $\dot{a}(\Z^d)$ from above. Another method was introduced by Eden \cite{Eden61} using more direct counting arguments. This method was enhanced by Klarner and Rivest \cite{KlaRive} in the case of $\Z^2$, who obtained that $\dot{a}(\Z^2)\leq 4.6496$, and more recently by Barequet and Shalah \cite{BareShal}, who obtained the asymptotic inequality $\dot{a}(\Z^d)\leq 2de-2e+1/(2d-2)$. In dimension $3$, the same paper proves $\dot{a}(\Z^3)<9.3835$. Plugging this into \eqref{bounds}, we deduce 
$\dot{p}_c(\Z^3)> 0.2522$, which is the best rigorous lower bound known.

\section{Preliminaries}

A \defi{lattice animal} $S$ is a connected subgraph of the hypercubic lattice $\Z^d$ containing a fixed vertex $o$. The (edge) boundary $\partial_E S$ of $S$ is the set of edges of $\Z^d$ that have at least one endvertex in $S$ but are not contained in $S$.  
Let $a_n(\Z^d)$ be the number of all lattice animals of $\Z^d$ with $n$ edges. 
It is well known that  $a(\Z^d):=\lim_{n\to\infty} {a_n(\Z^d)}^{1/n}$ exists 
\cite{Klarner,Klein}. 

A \defi{lattice site-animal} $S$ is a set of vertices of $\Z^d$ containing $o$ that spans a connected graph. Thus every lattice site-animal is a lattice animal. The (vertex) boundary $\partial_V S$ of $S$ is the set of vertices of $\Z^d$ that have a neighbour in $S$ but are not contained in $S$. 
Let $\dot{a}_n(\Z^d)$ be the number of all lattice site-animals of $\Z^d$ with $n$ vertices. We let $\dot{a}(\Z^d):=\lim_{n\to\infty} {\dot{a}_n(\Z^d)}^{1/n}$.


As already mentioned, we are interested in a sub-family of lattice\\ (site-)animals, called \defi{(site-)interfaces}, which we introduced in \cite{analyticity}, where they played a central role in proving many results about percolation. The intuition behind the notion is that $P\subset \Z^d$ is called an interface, if there is a percolation configuration in which the cluster $C_o$ of the origin is finite, and $P$ is the subgraph of $C_o$ separating it from infinity. The precise definition, which allows $P$ to be unambiguously determined by $C_o$, is rather involved, and can be found in \cite{analyticity} or the companion paper \cite{brI}. In the rest of this section we will recall the properties of interfaces that are relevant for this paper, so that the reader can follow our statements and proofs without the omitted technical details. We remark in passing that the definition of interfaces depends on the choice of a basis of the cycle space of $\Z^d$. When the full cycle space is chosen as a basis, for example, then lattice (site-)animals coincide with (site-)interfaces. But usually the basis we choose is the one consisting of all the 4-cycles of $\Z^d$, which leads to much thinner interfaces. To illustrate this point, we remark that for this choice of basis, interfaces satisfy the following geometric property in dimension $2$. Each edge of an interface $P$ is incident to the unbounded face of $P$, where now we view $P$ as a plane graph with its natural embedding inherited from $\Z^2$. In fact, in this specific case, interfaces can be defined as the set of those connected graphs that satisfy the latter property. 

\mymargin{removed: Lattice (site-)trees are a special species of (site-) interfaces, and the latter are a special species of lattice (site-)animals.}
Another important feature is that to each interface $P$ we associate a \defi{boundary $\partial P$}. Each edge in $\partial P$ has a common endvertex with some edge in $P$, but no edge in $\partial P$ is contained in $P$. In other words, $\partial P$ is contained in the set $\partial_E P$ defined above. However, it is possible that $\partial P$ is a proper subset of $\partial_E P$. To illustrate this, we remark that in dimension $2$ for example, if the basis of the cycle space we choose is the one consisting of all the 4-cycles of $\Z^2$, then $\partial P$ can be defined as the set of those edges of $\partial_E P$ that lie in the unbounded face of $P$. The precise definition $\partial P$ is again rather involved, and the interested reader can find it in \cite{analyticity} or \cite{brI}, but the above properties are all that we will need in this paper. Similarly, each site-interface $P$ has its own boundary, which is denoted for convenience by $\partial P$, and it is contained in $\partial_V P$. 

In analogy to the case of lattice animals and lattice site-animals, we define $c_n(\Z^d)$ and $\dot{c}_n(\Z^d)$ to be the number of interfaces and site-interfaces, respectively, of size $n$. Here the \defi{size} $|P|$ refers to the number of edges in the case of interfaces, and the number of vertices in the case of site-interfaces. Then we define $b(\Z^d):=\lim_{n\to\infty} {c_n(\Z^d)}^{1/n}$ and $\dot{b}(\Z^d):=\lim_{n\to\infty} {\dot{c}_n(\Z^d)}^{1/n}$.  As we observed in \cite{brI}, it is more fruitful to parametrize the exponential growth rate of (site-)interfaces according to their `volume-to-surface ratio'. For a possible `size' $n\in \N$, `volume-to-surface ratio' $r \in \R_+$, and `tolerance' $\epsilon \in \R_+$, we let $c_{n,r,\epsilon}(\Z^d)$ denote the number of interfaces $P$ with $|P|=n$ and $(r-\epsilon)n \leq |\partial P| \leq (r+\epsilon)n$. These numbers grow exponentially in $n$, and we define $b_r$ to be their exponential growth rate as $\epsilon \to 0$:
$$b_r=b_r(\Z^d):= \lim_{\epsilon \to 0} \limsup_{n\to\infty} {c_{n,r,\epsilon}(\Z^d)}^{1/n}.$$ 
The site variant $\dot{b}_r$ is defined analogously.
It is not hard to prove (see \cite[\Prr{equality}]{brI}) that 
\begin{equation}\label{equality}
b(\Z^d)=\max_r b_r(\Z^d).
\end{equation}

The function $b_r$ has strong implications for the behaviour of percolation on the lattice at hand. In particular, as we observed in \cite[Theorem 1.2]{brI}, one can determine whether the probability that an interface of size $n$ occurs in a percolation configuration of parameter $p$ decays exponentially by estimating the value of $b_r$. Indeed, one has the dimension-independent bounds 
\begin{equation}\label{b_r f_r ineq}
b_r\leq f(r)
\end{equation}
where $f(r):= \frac{(1+r)^{1+r}}{r^r}$, with equality if and only if the latter probability does not decay exponentially in $n$ for $p=p(r)=\frac{1}{1+r}$.

\medskip
This observation allows us to translate any upper bound on $\dot{p}_c(\Z^d)$ into a lower bound on $\dot{a}(\Z^d)$, and conversely any upper bound on $\dot{a}(\Z^d)$ into a lower bound on $\dot{p}_c(\Z^d)$.
Indeed, letting $r(p):= \frac{1-p}{p}$ (the inverse of the function $p(r)$ appearing above), we have \cite[Proposition 4.6]{brI}
\labtequ{bounds}{${a}(\Z^d) \geq {b}(\Z^d) \geq b_{r({p}_c)}(\Z^d) = f(r({p}_c(\Z^d)))$,} 
where the two inequalities are obvious from the definitions (\pint s are a species of lattice animal), and the equality is due to the fact that \eqref{b_r f_r ineq} holds with equality at $p_c$, i.e.\ for $r=r(p_c)$, as the aforementioned exponential decay fails there. To translate bounds on ${p}_c(\Z^d)$ into bounds on ${a}(\Z^d)$ or $b(\Z^d)$ and vice-versa, we just remark that $f(r)$ is monotone increasing in $r$, and $r(p)$ is monotone decreasing in $p$. Inequality \eqref{bounds} and the above reasoning applies verbatim to the site versions $\dot{p}_c(\Z^d)$ and $\dot{a}(\Z^d)$.

In two dimensions we cannot hope to get close to the real value of $\dot{a}(\Z^d)$ with this technique, as we are only enumerating the subspecies of site-\pint s\footnote{Still, for the hexagonal (aka. honeycomb) lattice $\mathbb{H}$, the best known lower bound was $\dot{a}(\mathbb{H})\geq 2.35$ \cite{BaRoShaImp,RanWelAni}, until this was recently improved to $\dot{a}(\mathbb{H})\geq 2.8424$ \cite{BaShaZheImp}. Plugging a numerical value for $\dot{p}_c(\mathbb{H})$, for which the most pessimistic (i.e.\ highest) estimate currently available is about $0.69704$ \cite{Jacob14}, we obtain $\dot{a}(\mathbb{H})\geq 2.41073$. If those approximations were rigorous, this would have improved the bounds of \cite{BaRoShaImp,RanWelAni}.}. But as we will see in the next section, our lower bounds become asymptotically tight as the dimension $d$ tends to infinity. In \Sr{sec upper} we will argue conversely: we will prove upper bounds on $\dot{a}(\Z^d)$ and plug them into \eqref{bounds} to obtain lower bounds on $\dot{p}_c(\Z^d)$.

\section{Bounds on growth rates of lattice animals and interfaces}\label{intro animals}

Our first result provides the first terms of the $1/d$ asymptotic expansion of the exponential growth rate of \pint s:

\begin{theorem}\label{asymptotics}
The exponential growth rate of the number of \pint s of $\Z^d$ satisfies $b(\Z^d)=2de-\dfrac{3e}{2}-O(1/d)$.
\end{theorem}
\begin{proof}
We claim that for any \pint\ $P$ of $\Z^d$ we have $|\partial P|\leq (2d-2)|P|+2d$.
Indeed, summing vertex degrees gives $\sum_{u\in V(P)} deg(u)\geq 2|P|+|\partial P|$, where $deg(u)$ is the degree of $u$ in the graph $P\cup \partial P$, because the edges of $P$ are counted twice, and the edges of $\partial P$ are counted at least once. Since $deg(u)\leq 2d$ and $|V(P)|\leq |P|+1$, we get
$$2|P|+|\partial P|\leq \sum_{u\in V(P)} deg(u)\leq 2d |V(P)|\leq 2d|P|+2d.$$ By rearranging we obtain the desired inequality. It follows that $b_r=0$ for every $r>2d-2$ which combined with \eqref{b_r f_r ineq} and the fact that $f(r)$ is an increasing function of $r$ gives 
$$b_r(\Z^d)\leq \dfrac{(2d-1)^{(2d-1)}}{(2d-2)^{(2d-2)}}$$
for $r\geq 0$. 
Using \eqref{equality} we obtain that
\begin{align} \label{bZ ineq}
b(\Z^d)\leq \dfrac{(2d-1)^{(2d-1)}}{(2d-2)^{(2d-2)}}.
\end{align}

Notice that for every $r>0$, $$\dfrac{(1+r)^{1+r}}{r^r}=(1+r)\Big(1+\dfrac{1}{r}\Big)^r=(1+r)\exp\Big(r\log\Big(1+\dfrac{1}{r}\Big)\Big).$$
Using the Taylor expansion $\log\Big(1+\dfrac{1}{r}\Big)=\dfrac{1}{r}-\dfrac{1}{2r^2}+\dfrac{1}{3r^3}-O(1/r^4)$ we obtain 
$$\dfrac{(1+r)^{1+r}}{r^r}=(1+r)\exp\Big(1-\dfrac{1}{2r}+\dfrac{1}{3r^2}-O(1/r^3)\Big)$$
as $r\to \infty$.
Now the Taylor expansion 
$$\exp(1+x)=e\Big(1+x+\dfrac{x^2}{2}+O(x^3)\Big)=e\Big(1-\dfrac{1}{2r}+\dfrac{11}{24r^2}-O(1/r^3)\Big),$$
where $x=-\dfrac{1}{2r}+\dfrac{1}{3r^2}-O(1/r^3)$, gives
\begin{gather*}
(1+r)\exp\Big(1-\dfrac{1}{2r}+\dfrac{1}{3r^2}-O(1/r^3)\Big)=(1+r)e\Big(1-\dfrac{1}{2r}+\dfrac{11}{24r^2}-O(1/r^3)\Big)= \\ er+\dfrac{e}{2}-O(1/r).
\end{gather*}
Consequently, 
\begin{align}\label{approx}
\dfrac{(1+r)^{1+r}}{r^r}=er+\dfrac{e}{2}-O(1/r).
\end{align}
Plugging $r=2d-2$ in \eqref{approx} we deduce that
\begin{align}\label{d-approx}
\dfrac{(2d-1)^{(2d-1)}}{(2d-2)^{(2d-2)}}=2de-3e/2-O(1/d),
\end{align}
which combined with \eqref{bZ ineq} establishes the desired upper bound.

For the lower bound, we have $b(\Z^d)\geq b_{r_d}(\Z^d)$ and $b_{r_d}(\Z^d)=f(r_d)$, where $r_d:=r(p_c(\Z^d))$. It has been proved in \cite{HaraSlade,HofstadSlade} that 
\begin{align}\label{critical}
p_c(\Z^d)=\dfrac{1}{2d}+\dfrac{1}{(2d)^2}+\dfrac{7}{2(2d)^3}+O(1/d^4),
\end{align} hence
$$r_d=\frac{1-p_c(\Z^d)}{p_c(\Z^d)}=\dfrac{16d^4}{8d^3+4d^2+7d+O(1)}-1.$$
We can easily compute that
\begin{gather*}
\dfrac{16d^4}{8d^3+4d^2+7d+O(1)}= 2d-\dfrac{8d^3+14d^2+O(d)}{8d^3+4d^2+7d+O(1)}=\\ 2d-\dfrac{8d^3+4d^2}{8d^3+4d^2+7d+O(1)}-O(1/d)
\end{gather*}
and
$$\dfrac{8d^3+4d^2}{8d^3+4d^2+7d+O(1)}=\dfrac{1}{1+O(1/d^2)}=1-O(1/d^2).$$
Hence $r_d=2d-2-O(1/d)$, which implies that $$b_{r_d}(\Z^d)=\frac{(1+r_d)^{1+r_d}}{r_d^{r_d}}=2de-3e/2-O(1/d).$$ 
Therefore, $b(\Z^d)=2de-\dfrac{3e}{2}-O(1/d)$ as desired.
\end{proof}

We remark that the asymptotic expansions of $\dfrac{(2d-1)^{(2d-1)}}{(2d-2)^{(2d-2)}}$ and $b_{r_d}$ differ in their third terms, and so we are unable to compute the third term in the asymptotic expansion of $b(\Z^d)$. It follows from the proof of \Tr{asymptotics} above that $b(\Z^d)-b_{r_d}(\Z^d)=O(1/d)$, i.e. $b_{r_d}$ is a good approximation of $b(\Z^d)$.

\medskip
Next, we use \Tr{asymptotics} and Kesten's argument \cite{Grimmett} to obtain the first two terms in the asymptotic expansion of $a(\Z^d)$. These had already been obtained by Miranda and Slade \cite{LatticeTreesTerms} but our proof is shorter. 

\begin{theorem}\label{animals}
$a(\Z^d)=2de-\dfrac{3e}{2}-O(1/d)$.
\end{theorem}
\begin{proof}
Let $C$ be a connected subgraph containing $o$. Arguing as in the proof of \Tr{asymptotics}, we obtain that $|\partial_E C|\leq (2d-2)|E(C)|+2d$. It follows that for every $p\in (0,1)$.
$$a_n(\Z^d)p^n (1-p)^{(2d-2)n+2d} \leq \Pr_p(|E(C_o)|=n)\leq 1.$$
Choosing $p=\frac{1}{2d-1}$ and dividing by $p^n (1-p)^{(2d-2)n+2d}$, we deduce from \eqref{d-approx} that $$a(\Z^d)\leq \frac{(2d-1)^{(2d-1)}}{(2d-2)^{(2d-2)}}=2de-3e/2-O(1/d).$$
Since $a(\Z^d)\geq b(\Z^d)$, the lower bound follows from \Tr{asymptotics}.
\end{proof}

The behaviour of $a(\Z^d)$, and the analogue $t(\Z^d)$ for lattice-trees, has been extensively studied in the physics literature. The expansions 
\begin{align*}
a(\Z^d)=\sigma e\exp &\left(- \frac 12 \frac{1}{\sigma}- \big(\frac{8}{3}-\frac{1}{2e}\big) \frac{1}{\sigma^2}- \big(\frac{85}{12}-\frac{1}{4e}\big) \frac{1}{\sigma^3}- \big(\frac{931}{20}-\frac{139}{48e} - \frac{1}{8e^2}\big)\frac{1}{\sigma^4}\right.\nonumber \\ & \hspace{5mm}\left. - \big(\frac{2777}{10}+\frac{177}{32e} - \frac{29}{12e^2}\big)\frac{1}{\sigma^5} + \cdots \right)
\end{align*}
and
\begin{align}\label{trees}
t(\Z^d)=\sigma e\exp\left(- \frac 12 \frac{1}{\sigma}- \frac 83 \frac{1}{\sigma^2}- \frac {85}{12} \frac{1}{\sigma^3}- \frac {931}{20} \frac{1}{\sigma^4}- \frac {2777}{10} \frac{1}{\sigma^5}+ \cdots\right),
\end{align}
where $\sigma=2d-1$, were reported in \cite{GP}, \cite{Harr82, PG95}, respectively, but without any rigorous bounds on the error terms. Miranda and Slade \cite{LatticeTrees} proved that both $a(\Z^d)$ and $t(\Z^d)$ are asymptotic to $2de$. The first three terms of $a(\Z^d)$ and $t(\Z^d)$ have been computed rigorously by the same authors in \cite{LatticeTreesTerms}.

Any lattice tree is an \pint , and therefore we have $t(\Z^d)\leq b(\Z^d)\leq a(\Z^d)$. Although the first two terms in the asymptotic expansions of each of them are the same, we expect that the strict inequality $t(\Z^d)<b(\Z^d)$ holds (independently of the choice of a basis of the cycle space used to define interfaces). The strict inequality $b(\Z^d)<a(\Z^d)$ is proved in the companion paper \cite{brI}, when the chosen basis of the cycle space contains only cycles of bounded length, using an argument similar to that in the proof of Kesten's pattern theorem for self-avoiding walks \cite{KestenI}. Proving the inequality $t(\Z^d)<b(\Z^d)$ seems more challenging because even a local modification on the structure of a lattice tree can have global effects on the structure of the corresponding interface. 


We recall that using \eqref{critical} we can easily compute the first three terms of the $1/d$ expansion of $b_{r_d}(\Z^d)$, from which we obtain a lower bound on $b(\Z^d)$, but only the first two of them coincide with the corresponding terms of the upper bound. On the other hand, we can check that all first three terms of the $1/d$ expansion of $b_{r_d}(\Z^d)$ coincide with the corresponding terms of the $1/d$ expansion of $t(\Z^d)$. 
However, we expect that the fourth term of the asymptotic expansion of $b_{r_d}(\Z^d)$ is strictly smaller than the fourth term of the asymptotic expansion of $t(\Z^d)$, as suggested by \eqref{trees} and the asymptotic expansion 
$$p_c(\Z^d)=\dfrac{1}{\sigma}+\dfrac{5}{2\sigma^3}+\dfrac{15}{2\sigma^4}+\dfrac{57}{\sigma^5}+\cdots$$ that is reported in \cite{GauntRuskin} without rigorous proof. This implies the strict inequalities $b_{r_d}(\Z^d)<t(\Z^d)$ and $b_{r_d}(\Z^d)<b(\Z^d)$ for every large enough value of $d$. We expect that these strict inequalities hold for every $d>1$. For example, we know that $b_{r_2}(\Z^2)=4$, because $p_c(\Z^2)=1/2$ \cite{KestenCritical}. On the other hand, for small enough numbers $n$, the value of $t_n(\Z^2)$ is known exactly, and a concatenation argument yields the lower bound $t(\Z^2)\geq 4.1507$ \cite{GSTW82,WS90}.

\subsection{Site variants}
We now prove analogous results for site-interfaces and site-animals. We start with a weaker analogue of \Tr{asymptotics}:

\begin{theorem}\label{a Zd}
The exponential growth rate of the number of site-\pint s of $\Z^d$ satisfies $\dot{b}(\Z^d)=2de-O(1)$.
\end{theorem}
\begin{proof}
Similarly to the proof of \Tr{asymptotics}, we will show that for any site-\pint\ $P$ of $\Z^d$ we have $|\partial P|\leq (2d-2)|P|+2$.
Let $k$ be the number of edges of the graph spanned by $P$, and let $l$ be the number of edges with one endvertex in $P$ and one in $\partial P$. Notice that $k\geq |P|-1$ and $l\geq |\partial P|$. Arguing as in the proof of \Tr{asymptotics} we obtain 
$$2(|P|-1)+|\partial P|\leq 2k+l\leq 2d|P|.$$ By rearranging we obtain the desired inequality. Arguing as in the proof of \Tr{asymptotics} we obtain
$$\dot{b}(\Z^d)\leq \frac{(2d-1)^{(2d-1)}}{(2d-2)^{(2d-2)}}=2de-O(1).$$
Moreover, we have that $\dot{b}(\Z^d)\geq \dot{b}_{\dot{r}_d}(\Z^d)$ and $\dot{b}_{\dot{r}_d}(\Z^d)=f(\dot{r}_d)$, where $\dot{r}_d:=r(\pcs(\Z^d))$. Hara and Slade \cite{HaraSlade} proved that $\pcs(\Z^d)=\big(1+O(1/d)\big)/2d$, hence $$\dot{r}_d=\frac{1-\pcs(\Z^d)}{\pcs(\Z^d)}=\dfrac{2d}{1+O(1/d)}-1.$$ Using \eqref{approx} we obtain 
$$\dot{b}_{\dot{r}_d}(\Z^d)=\frac{(1+\dot{r}_d)^{1+\dot{r}_d}}{\dot{r}_d^{\dot{r}_d}}=\dfrac{2de}{1+O(1/d)}-e/2-O(1/d).$$
Since $\dfrac{1}{1+O(1/d)}=1-O(1/d)$, we have $$\dfrac{2de}{1+O(1/d)}-e/2-O(1/d)=2de\Big(1-O(1/d)\Big)-e/2-O(1/d)=2de-O(1).$$
Therefore, $\dot{b}_{\dot{r}_d}(\Z^d)=2de-O(1)$, which implies that $\dot{b}(\Z^d)=2de-O(1)$ as desired.
\end{proof}

In the previous section we used \eqref{bounds} and \eqref{critical} to lower-bound $b(\Z^d)$. 
Recently, Heydenreich and Matzke \cite{HeyMat} proved that\footnote{We remark that the more detailed expansion 
\begin{align}\label{site perco asym}
\pcs(\Z^d)=\dfrac{1}{\sigma}+\dfrac{3}{2\sigma^2}+\dfrac{15}{4\sigma^3}+\dfrac{83}{4\sigma^4}+\cdots
\end{align}
was reported in \cite{GSR76} without any rigorous bounds on the error terms. } 
\begin{align}\label{site perco asym rig}
\pcs(\Z^d)=\dfrac{1}{2d}+\dfrac{5}{2(2d)^2}+\dfrac{31}{4(2d)^3}+O(1/d^4).
\end{align}

Combining \eqref{site perco asym rig} with our above method gives the lower bound $\dot{a}(\Z^d)\geq \dot{b}(\Z^d)\geq 2de-3e+O(1/d)$. Arguing as in \Tr{animals}, we can easily obtain 
\begin{theorem}\label{gap}
$\dot{a}(\Z^d)\leq 2de-O(1)$ and $\dot{a}(\Z^d)\geq 2de-3e+O(1/d)$. 
\end{theorem}
Barequet, Barequet and Rote  \cite{SiteAnimals} proved the weaker result $\dot{a}(\Z^d)=2de-o(d)$, and they conjectured that $\dot{a}(\Z^d)=2de-3e+O(1/d)$ in agreement with physicists' predictions \cite[(2.22)]{PG95}, so it only remains to prove a matching upper bound\footnote{In fact \cite{SiteAnimals} offers the more detailed conjecture $\dot{a}(\Z^d)=2de-3e-\frac{31e}{48d}+O(1/d^2)$.}. We will improve the upper bound in \Tr{improved bounds} below.
We remark that under the assumption $\dot{a}(\Z^d)=2de-3e+O(1/d)$ holds, we obtain $\dot{b}(\Z^d)-\dot{b}_{\dot{r}_d}(\Z^d)=O(1/d)$.

\section{Upper bounds for lattice site-animals} \label{sec upper}

In the previous section we used Kesten's argument in order to upper bound $\dot{a}(\Z^d)$. Another method that gives the same upper bounds for $\dot{a}(\Z^d)$ was introduced by Eden \cite{Eden61}. Eden described a procedure that associates in a canonical way, a spanning tree and a binary sequence to every lattice site-animal. This reduces the problem of counting lattice site-animals to a problem of counting binary sequences with certain properties. Klarner and Rivest \cite{KlaRive} enhanced Eden's method in the case of $\Z^2$, proving that $\dot{a}(\Z^2)\leq 4.6496$. Recently, Barequet and Shalah \cite{BareShal} extended this enhancement to higher dimensions, obtaining $\dot{a}(\Z^d)\leq 2de-2e+1/(2d-2)$.

In this section we will utilise Eden's procedure to reduce the gap between the aforementioned inequality and the conjectured asymptotic expansion $\dot{a}(\Z^d)=2de-3e+O(1/d)$ mentioned in the previous section:

\begin{theorem}\label{improved bounds}
We have $\dot{a}(\Z^d)\leq 2de-5e/2+O(1/\log(d))$. 
\end{theorem}

Our result improves the bounds of Barequet and Shalah \cite{BareShal} for every large enough $d$. 

We remark that $\dot{b}_{r_d}(\Z^d)=2de-3e+O(1/d)$ by \eqref{site perco asym rig}.
It is reasonable to expect that both $\dot{b}(\Z^d)-\dot{b}_{r_d}(\Z^d)=O(1/d)$ and $\dot{a}(\Z^d) - \dot{b}(\Z^d)=O(1/d)$ hold, as it happens for the bond variants, which would imply the aforementioned conjecture $\dot{a}(\Z^d)=2de-3e+O(1/d)$.

In order to prove \Tr{improved bounds}, we will show that a typical lattice site-animal has volume-to-surface ratio that is bounded away from its maximal possible value, namely $2d-2$. 

We will need the following definition. We let $\dot{a}_{n,r,\epsilon}$ denote the number of lattice site-animals $X$ of $\Z^d$ containing $o$ with $|X|=n$ and $(r-\epsilon)n \leq |\partial_V X| \leq (r+\epsilon)n$, and we define 
$$\dot{a}_r=\dot{a}_r(\Z^d):= \lim_{\epsilon \to 0} \limsup_{n\to\infty} {\dot{a}_{n,r,\epsilon}(\Z^d)}^{1/n}.$$
Using Kesten's argument, one can show that 
\labtequ{a leq f}{$\dot{a}_r \leq f(r)$.}
for every $r>0$. This follows from the work of Hammond \cite{HammondExpRates}, and it can also be seen as a special case of \eqref{b_r f_r ineq}, since by choosing the full cycle space of $\Z^d$ as its basis, each lattice site-animal $P$ is a site-interface with $\partial P=\partial_V P$.

For the proof of \Tr{improved bounds} we will need the next lemma which bounds $\dot{a}_r(\Z^d)$ for $r$ close to $2d-2$. We remark that $\dot{a}_{2d-2}(\Z^d)\geq \dot{b}_{2d-2}(\Z^d)\geq 1$, as a straight path has volume-to-surface ratio roughly $2d-2$. We also make the convention $0^0=1$.

\begin{lemma}\label{improved bounds r}
Consider some $0\leq x \leq 1$, and let $y=\min\{x,1/2\}$. Then 
$$\dot{a}_{2d-2-x}(\Z^d) \leq \dfrac{(2d-1)^{2d-1}}{y^y (1-y)^{1-y} x^x (2d-1-x)^{2d-1-x}}.$$ 
In particular, $\dot{a}_{2d-2}(\Z^d)=1$.
\end{lemma}
\begin{proof}
For $x=1$ we have $y=1/2$, and so the claimed upper bound is equal to 
$$2\dfrac{(2d-1)^{2d-1}}{(2d-2)^{2d-2}},$$
which is in turn equal to $2f(2d-2)$. Since $f(r)$ is an increasing function, 
$$f(2d-3)\leq f(2d-2)\leq 2f(2d-2).$$
The assertion now follows in the case $x=1$ from the fact that $\dot{a}_{2d-3}(\Z^d)\leq f(2d-3)$. So let us assume that $x<1$.

Let us start by introducing some necessary definitions. The \defi{lexicographical ordering} of $\Z^d$ is defined as follows. We say that a vertex $u=(u_1,u_2,\ldots,u_d)$ is smaller than a vertex $v=(v_1,v_2,\ldots,v_d)$ if there is some $i=1,2,\ldots,d$ such that $u_i\leq v_i$ and $u_j=v_j$ for every $j<i$. We also order the directed edges of the form $\ar{ou}$ in an arbitrary way. The latter ordering induces by translation a natural ordering of the set of directed edges with a common initial endvertex $v$, where $v$ is any vertex of $\Z^d$.

Consider some numbers $n\in \mathbb{N}$, and $\epsilon>0$ with $x+\epsilon<1$. We will start by describing Eden's procedure. Let $X$ be a lattice site-animal of size $n$ in $\Z^d$ containing $o$, such that $(2d-2-x-\epsilon)n\leq |\partial_V X|\leq (2d-2-x+\epsilon)n$. We will assign to $X$ a unique binary sequence $S=S(X)=(s_1,s_2,\ldots,s_{(2d-1)n-d+1})$ of length $(2d-1)n-d+1$. To this end, we will reveal the vertices of $X$ one by one in a specific way. Let $v_1$ be the lexicographically smallest vertex of $X$, and notice that $v_1$ has at most $d$ neighbours in $X$. For every $i=1,\ldots,d$, we let $s_i$ take the value $1$ if the $i$th directed edge of the form $\ar{u_1v}$ in the above ordering lies in the set of directed edges $\overleftrightarrow{E(X)}$ of $X$, and $0$ otherwise. The ordering of these directed edges induces an ordering on the neighbours of $u_1$ in $P$. We reveal the neighbours of $u_1$ in $X$ one by one according to the latter ordering, and we let $u_{j+1}$ be the $j$th revealed vertex. Now we proceed to the lexicographically smaller neighbour of $u_1$ lying in $X$, denoted $w$. The valid directed edges starting from $w$ are those not ending at $u_1$, and there are exactly $2d-1$ of them. The ordering of the whole set of directed edges starting from $w$ induces an ordering of the set of valid directed edges starting from $w$. For every $i=d+1,\ldots,3d-1$, we let $s_i$ take the value $1$ if the $(i-d)$th valid directed edge of the form $\ar{wv}$ lies in $\overleftrightarrow{E(X)}$ and $v$ has not been revealed so far (the latter is always true in this step but not necessarily in the following steps), and $0$ otherwise. We reveal the corresponding neighbours of $w$ in $X$ one by one, and we label them $u_k,u_{k+1}\ldots,$ where $k$ is the smallest index not previously used. Now we proceed as before up to the point that all vertices of $X$ have been revealed, and we set to $0$ all the remaining entries of $S$ that have not already been set to some value. Notice that $S$ contains exactly $n-1$ $1$'s, since $P$ has size $n$.

The above construction defines naturally a spanning subtree $T$ of $X$ rooted at $u_1$, by attaching an edge $u_k u_l$, $k<l$ to $T$ when $u_l$ is one of the neighbours of $u_k$ revealed when considering the valid directed edges starting from $u_k$. Given an edge $uv$ of $T$ with $u$ being the ancestor of $v$, we say that $uv$ is a \defi{turn} of $T$ if $uv$ is perpendicular to the edge $zu$ of $T$, where $z$ is the (unique) ancestor of $u$. We denote by $t$ the number of turns of $T$. We claim that 
\begin{align}\label{turns}
|\partial_V X|\leq (2d-2)n-t+2.
\end{align} 
Indeed, for every $k=1,2,\ldots,n$, let $T_k$ be the subtree of $T$ with $V(T_k)=\{u_1,u_2,\ldots,u_k\}$. Let also $\partial T_k$ be the set of vertices in $\Z^d\setminus \{u_1,u_2,\ldots,u_k\}$ having a neighbour in $\{u_1,u_2,\ldots,u_k\}$. Write 
$t_k$ for the number of turns of $T_k$. We will prove inductively that $$|\partial T_k|\leq (2d-2)|T_k|-t_k+2$$
for every $k=1,2,\ldots,n$.
The claim will then follow once we observe that $|\partial_V X|= |\partial T_n|$, $|X|=|T_n|=n$ and $t=t_n$.
For $k=1$, the assertion clearly holds. Assume that it holds for some $1\leq k<n$. Notice that we always have $|T_{k+1}|=|T_k|+1$ and $|\partial T_{k+1}|\leq |\partial T_k|+2d-2$, because $u_{k+1}$ lies in $\partial T_k$ and at most $2d-1$ neighbours of $u_{k+1}$ lie in $\partial T_{k+1}$. If $t_{k+1}=t_k$, then we get $|\partial T_{k+1}|\leq (2d-2)|T_{k+1}|-t_{k+1}+2$, as claimed. Suppose that $t_{k+1}=t_k +1$. Consider the ancestor $u_l$ of $u_{k+1}$, and the ancestor $u_m$ of $u_l$. Since by adding $u_{k+1}$ to $T_k$ we create one more turn, $u_{k+1}$, $u_l$ and $u_m$ are three vertices of a common square. Let $w$ be the fourth vertex. Notice that $w$ lies in $T_k \cup \partial T_k$. Thus, at most $2d-2$ neighbours of $u_{k+1}$ lie in $\partial T_{k+1} \setminus \partial T_k$. Therefore, $|\partial T_{k+1}|\leq (2d-2)|T_{k+1}|-t_{k+1}+2$, as desired. This completes the proof of \eqref{turns}.

We will now utilise \eqref{turns} to prove the statement of the lemma. Our assumption $(2d-2-x-\epsilon)n\leq |\partial_V X|$ combined with \eqref{turns} implies that $t\leq (x+\epsilon)n+2$. Hence it suffices to find an upper bound for the number of lattice site-animals $Q$ of size $n$ with $t\leq q:=(x+\epsilon)n+2$. We claim that the number $\dot{a}_n$ of such lattice site-animals of size $n$ satisfies 
\begin{align}\label{ijq}
\dot{a}_n\leq\sum_{i=1}^d \sum_{j=0}^{\min\{q,n-i\}} {d \choose i}{(2d-1)(n-1) \choose j} {n-1 \choose n-i-j}.
\end{align}
Indeed, let $i$ be number of neighbours of $u_1$ in $Q$, let $j$ be the number of $1$'s contributing to the number of turns in those bits of $S(Q)$. Let us apply the following steps in turn:
\begin{enumerate}
\item Set $i$ entries of $(s_1,\ldots,s_d)$ equal to $1$,
\item Choose which entries of $S(Q)$ contribute to the number of turns,
\item Choose which bits, except for the first one, contain an additional $1$.
\end{enumerate}
After the first two steps, we have specified which entries of $S(Q)$ are set to $1$, except for those that do not contribute to the number of turns. Since for every vertex of $Q$, at most one of its children does not contribute to the number of turns, we conclude that at most one entry of each of the bits chosen in the fourth step can be set to $1$, the position of which in $S(Q)$ is uniquely determined by the values of the remaining entries of $S(Q)$. It is easy to see now that for every $i$ and $j$, there are at most 
$${d \choose i}{(2d-1)(n-1) \choose j} {n-1 \choose n-i-j}$$
possibilities for $Q$, and so \eqref{ijq} can be obtained by summing over all possible values of $i$ and $j$.

We will now handle the sum in the right-hand side of \eqref{ijq}. Since the binomial coefficient ${m \choose l}$ is an increasing function of $l$ when $l\leq m/2$, we have $${(2d-1)(n-1) \choose j}\leq {(2d-1)(n-1) \choose q}.$$ Using Stirling's approximation $m!=\big(1+o(1)\big) \sqrt{2\pi m}(m/e)^m$ we obtain 
$${(2d-1)(n-1) \choose q}\approx \dfrac{(2d-1)^{(2d-1)n}}{(x+\epsilon)^{x+\epsilon}(2d-1-x-\epsilon)^{(2d-1-x-\epsilon)n}},$$
where $\approx$ denotes equality up to a multiplicative constant that is $O(c^n)$ for every $c>1$. 
Clearly $${n-1 \choose n-i-j}\leq 2^n.$$ 
It follows that
$$\dot{a}_{n,2d-2-x,\epsilon} \lesssim 2^n \dfrac{(2d-1)^{(2d-1)n}}{(x+\epsilon)^{x+\epsilon}(2d-1-x-\epsilon)^{(2d-1-x-\epsilon)n}},$$
where $\lesssim$ denotes inequality up to a multiplicative constant that is $O(c^n)$ for every $c>1$.
Taking $n$th roots and letting $n\to\infty$ and $\epsilon\to 0$ we obtain $$\dot{a}_{2d-2-x}\leq 2 \dfrac{(2d-1)^{2d-1}}{x^x(2d-1-x)^{2d-1-x}}.$$

The above bound can be improved when $x<1/2$. Suppose that $x<1/2$. We can choose $\epsilon>0$ small enough, and increase the value of $n$, if necessary, to ensure that $q+d< n/2$. Since the binomial coefficient ${m \choose l}$ is a decreasing function of $l$ when $l\geq m/2$, for every $i$ and $j$, we have
$${n-1 \choose n-i-j}\leq {n-1 \choose n-d-q},$$
because $n-i-j\geq n-d-q \geq n/2$.
Using again Stirling's approximation, we deduce that
$${n-1 \choose n-d-q}\approx\big((x+\epsilon)^{x+\epsilon}(1-x-\epsilon)^{1-x-\epsilon}\big)^{-n}.$$
We can now conclude that
$$\dot{a}_{n,2d-2-x,\epsilon}\lesssim \dfrac{(2d-1)^{(2d-1)n}}{(x+\epsilon)^{(2x+2\epsilon)n} (1-x-\epsilon)^{(1-x-\epsilon)n}(2d-1-x)^{(2d-1-x)n}}.$$ Taking $n$th roots and letting $n\to\infty$ and $\epsilon\to 0$ we obtain $$\dot{a}_{2d-2-x}\leq \dfrac{(2d-1)^{2d-1}}{x^{2x} (1-x)^{1-x}(2d-1-x)^{2d-1-x}}.$$
\end{proof}

Since a site-\pint\ is also a lattice site-animal and $\partial P\subset \partial_V P$, we obtain

\begin{corollary} \label{improved bounds cor}
Consider some $0\leq x \leq 1$, and let $y=\min\{x,1/2\}$. Then 
$$\dot{b}_{2d-2-x}(\Z^d) \leq \dfrac{(2d-1)^{2d-1}}{y^y (1-y)^{1-y} x^x (2d-1-x)^{2d-1-x}}.$$ 
In particular, $\dot{b}_{2d-2}(\Z^d)=1$.
\end{corollary}

\medskip

We are now ready to prove \Tr{improved bounds}.
\begin{proof}[Proof of \Tr{improved bounds}]
For every $0\leq x\leq 1$, we let $$g_d(x)=\dfrac{(2d-1)^{2d-1}}{y^y (1-y)^{1-y} x^x (2d-1-x)^{2d-1-x}},$$
where $y=\min\{x,1/2\}$.
It is not hard to see that there is a constant $C>0$ such that $x^{-x}\leq C$ for every $x\in [0,1]$, and 
$$\dfrac{1}{y^y(1-y)^{1-y}}\leq C$$
for every $y\in [0,1/2]$.
Moreover, for every $x\in [0,1]$ we have
$$\dfrac{(2d-1)^{2d-1}}{(2d-1-x)^{2d-1-x}}\leq \dfrac{(2d-1)^{2d-1}}{(2d-2)^{2d-1-x}}$$ 
by the monotonicity of $2d-1-x$ as a function of $x$,
and
$$\dfrac{(2d-1)^{2d-1}}{(2d-2)^{2d-1-x}}=\dfrac{2d-1}{(2d-2)^{1-x}}\Big(1+\dfrac{1}{2d-2}\Big)^{2d-2}\leq \dfrac{2d-1}{(2d-2)^{1-x}}e.$$
Thus, 
$$g_d(x)\leq C^2 e \dfrac{2d-1}{(2d-2)^{1-x}}.$$
Since $\dfrac{2d-1}{(2d-2)^{1-x}}$ is an increasing function of $x$, it follows by \Lr{improved bounds r} that for every $$x\leq z:=1-\dfrac{C^2}{\log\big(2d-2\big)}$$ we have
$$\dot{a}_{2d-2-x}(\Z^d)\leq g_d(x)\leq C^2 e\dfrac{2d-1}{(2d-2)^{1-x}}\leq C^2 e\dfrac{2d-1}{(2d-2)^{1-z}}=\\ C^2e^{1-C^2}(2d-1).$$
Using the standard inequality $e^{C^2}\geq 1+ C^2$ we obtain $e^{-C^2}\leq 1/(1+C^2)$, hence $$C^2e^{1-C^2}(2d-1)\leq \dfrac{C^2 e}{1+C^2}(2d-1).$$
Plugging $r=2d-2-z$ in \eqref{approx} we obtain $f(2d-2-z)=2de-5e/2+O(1/\log(d))$, and so 
\begin{align}\label{dot a}
\dot{a}_{2d-2-x}(\Z^d)< f(2d-2-z)
\end{align}
for every $d$ large enough.
On the other hand, for every $r\leq 2d-2-z$ we have $\dot{a}_r(\Z^d)\leq f(2d-2-z)$ by \eqref{a leq f}, \mymargin{this is where we use it} hence
$$\dot{a}(\Z^d)\leq f(2d-2-z)=2de-5e/2+O(1/\log(d))$$ by \eqref{equality} for every $d$ large enough (recall that lattice site-animals coincide with site-interfaces for a special choice of a basis of the cycle space), which proves our claim.
\end{proof}

Combining \Tr{improved bounds} with \eqref{bounds} yields the following lower bound for $\pcs(\Z^d)$:

\begin{theorem} \label{thm pc lower}
$\pcs(\Z^d)\geq \dfrac{1}{2d}+\dfrac{2}{(2d)^2}-O(1/d^2\log(d))$.
\end{theorem}
\begin{proof}
It follows from \eqref{dot a} that $b_r< f(2d-2-z)\leq f(r)$ for every $r\geq 2d-2-z$, where $z=1-\dfrac{C^2}{\log\big(2d-2\big)}$ and $C$ is the constant in the proof of \Tr{improved bounds}. Since 
$b_{\dot{r}_d}(\Z^d)=f(\dot{r}_d)$, we obtain 
$$\dot{r}_d\leq 2d-3+\dfrac{C^2}{\log\big(2d-2\big)}.$$ Hence
$$\pcs(\Z^d)=\dfrac{1}{1+\dot{r}_d}\geq \dfrac{1}{2d-2+C^2/\log(2d-2)}.$$
It is not hard to see 
\begin{gather*}
\dfrac{1}{2d-2+C^2/\log(2d-2)}=\dfrac{1}{2d}+\dfrac{2-C^2/\log(2d-2)}{2d\big(2d-2+C^2/\log(2d-2)\big)}=\\ \dfrac{1}{2d}+\dfrac{2}{(2d)^2}-O(1/d^2\log(d)),
\end{gather*}
which proves the assertion.
\end{proof}

We remark that the well known inequality $\pcs(\Z^d)\geq p_c(\Z^d)$ \cite{Grimmett} and the asymptotic expansion $p_c(\Z^d)=\dfrac{1}{2d}+\dfrac{1}{(2d)^2}+O(1/d^3)$, mentioned in the previous section, give a weaker lower bound on $\pcs(\Z^d)$.

Recently, Barequet and Shalah \cite{BareShal} proved that $\dot{a}(\Z^3)<9.3835$. Plugging this into \eqref{bounds}, we deduce 
\labtequ{pc 3}{$\dot{p}_c(\Z^3)> r^{-1} \circ f^{-1}(9.3835) > 0.2522$.}
As far as we know, the best rigorous bound previously known was about $\dot{p}_c(\Z^3)>0.21225$, obtained as the inverse of the best known bound on the connective constant \cite{SAWCubic}\footnote{We thank John Wierman for this remark.}. 

\medskip
{\bf Remark:} In both \Tr{thm pc lower} and \eqref{pc 3} we made implicit use of \eqref{b_r f_r ineq}, but it would have sufficed to use its variant for site-lattice animals instead of interfaces. Thus adapting Delyon's \cite{DelTai} result to site-animals would have sufficed.

\bibliographystyle{plain}
\bibliography{collective}

\begin{thebibliography}{10}

\bibitem{AlekBar}
G.~Aleksandrowicz and G.~Barequet.
\newblock The growth rate of high-dimensional tree polycubes.
\newblock {\em Electronic Notes in Discrete Mathematics}, 38:25--30, 2011.

\bibitem{SiteAnimals}
G.~Barequet, R.~Barequet, and G.~Rote.
\newblock {Formulae and growth rates of high-dimensional polycubes}.
\newblock {\em Combinatorica}, 30(3):257--275, 2010.

\bibitem{BaRoShaImp}
G.~Barequet, G.~Rote, and M.~Shalah.
\newblock {An improved upper bound on the growth constant of polyiamonds}.
\newblock {\em {Acta Mathematica Universitatis Comenianae}}, 88(3):429--436,
  2019.

\bibitem{BareShal}
G.~Barequet and M.~Shalah.
\newblock Improved upper bounds on the growth constants of polyominoes and
  polycubes.
\newblock In {\em Proc.\ 14th Latin American Theoretical Informatics Symposium,
  S\~{a}o Paolo, Brazil. Lecture Notes in Computer Science, Springer}, volume
  12118, pages 532--545, 2021.

\bibitem{BaShaZheImp}
G.~Barequet, M.~Shalah, and Y.~Zheng.
\newblock {An improved lower bound on the growth constant of polyiamonds}.
\newblock {\em Journal of Combinatorial Optimization}, 37(2):424--438, 2019.

\bibitem{DelTai}
F.~Delyon.
\newblock {\em {Taille, forme et nombre des amas dans les problemes de
  percolation, These de 3eme cycle}}.
\newblock {Universite Pierre et Marie Curie, Paris}, 1980.

\bibitem{Eden61}
M.~Eden.
\newblock {A Two-dimensional Growth Process}.
\newblock In {\em Proceedings of the Fourth Berkeley Symposium on Mathematical
  Statistics and Probability}, volume~4, pages 223--239, 1961.

\bibitem{GP}
D.~S. Gaunt and P.~J. Peard.
\newblock {$1/d$-expansions for the free energy of weakly embedded site animal
  models of branched polymers}.
\newblock {\em Journal of Physics A: Mathematical and General},
  33(42):7515--7539, 2000.

\bibitem{GauntRuskin}
D.~S. Gaunt and H.~Ruskin.
\newblock {Bond percolation processes in d dimensions}.
\newblock {\em Journal of Physics A: Mathematical and General},
  11(7):1369--1380, 1978.

\bibitem{GSR76}
D.~S. Gaunt, M.~F. Sykes, and H.~Ruskin.
\newblock Percolation processes in d-dimensions.
\newblock {\em Journal of Physics A: Mathematical and General},
  9(11):1899--1911, 1976.

\bibitem{GSTW82}
D.~S. Gaunt, M.~F. Sykes, G.M. Torrie, and S.~G. Whittington.
\newblock Universality in branched polymers on d-dimensional hypercubic
  lattices.
\newblock {\em Journal of Physics A: Mathematical and General},
  15(10):3209--3217, 1982.

\bibitem{analyticity}
A.~Georgakopoulos and C.~Panagiotis.
\newblock {Analyticity results in Bernoulli Percolation}.
\newblock To appear in Memoirs of the AMS.

\bibitem{brI}
A.~Georgakopoulos and C.~Panagiotis.
\newblock {On the exponential growth rates of lattice animals and interfaces
  I}.
\newblock arXiv:1908.03426.

\bibitem{Grimmett}
Geoffrey Grimmett.
\newblock {\em {Percolation, \emph{{Second Edition}}}}.
\newblock Grundlehren der mathematischen Wissenschaften. Springer, 1999.

\bibitem{HammondExpRates}
A.~Hammond.
\newblock Critical exponents in percolation via lattice animals.
\newblock {\em Electronic Communications in Probability}, 10:45--59, 2005.

\bibitem{HaraSlade}
T.~Hara and G.~Slade.
\newblock {The self-avoiding-walk and percolation critical points in high
  dimensions}.
\newblock {\em Combinatorics, Probability and Computing}, 4(3):197--215, 1995.

\bibitem{Harr82}
A.~B. Harris.
\newblock {Renormalized $(1/\sigma)$ expansion for lattice trees and
  localization}.
\newblock {\em Physical Review B}, 26(1):337--366, 1982.

\bibitem{HeyMat}
M.~Heydenreich and K.~Matzke.
\newblock Expansion for the critical point of site percolation: the first three
  terms.
\newblock arXiv:1912.04584.

\bibitem{HofstadSlade}
R.~Van~Der Hofstad and G.~Slade.
\newblock {Expansion in $n^{-1}$ for Percolation Critical Values on the
  $n$-cube and $\mathbb{Z}^n$: the First Three Terms}.
\newblock {\em Combinatorics, Probability and Computing}, 15(5):695--713, 2006.

\bibitem{Jacob14}
J.~L. Jacobsen.
\newblock {High-precision percolation thresholds and Potts-model critical
  manifolds from graph polynomials}.
\newblock {\em Journal of Physics A: Mathematical and Theoretical},
  47(13):135001+78, 2014.

\bibitem{KestenI}
H.~Kesten.
\newblock {On the Number of Self-Avoiding Walks}.
\newblock {\em Journal of Mathematical Physics}, 4(7):960--969, 1963.

\bibitem{KestenCritical}
H.~Kesten.
\newblock {The critical probability of bond percolation on the square lattice
  equals 1/2}.
\newblock {\em Communications in Mathematical Physics}, 74(1):41--59, 1980.

\bibitem{Klarner}
D.~A. Klarner.
\newblock {Cell growth problems}.
\newblock {\em Canadian Journal of Mathematics}, 19:851--863, 1967.

\bibitem{KlaRive}
D.~A. Klarner and R.~L. Rivest.
\newblock A procedure for improving the upper bound for the number of
  n-ominoes.
\newblock {\em Canadian Journal of Mathematics}, 25(3):585--602, 1973.

\bibitem{Klein}
D.J. Klein.
\newblock {Rigorous results for branched polymer models with excluded volume}.
\newblock {\em The Journal of Chemical Physics}, 75(10):5186--5189, 1981.

\bibitem{SAWCubic}
D.~MacDonald, S.~Joseph, D.~L. Hunter, L.~L. Moseley, N.~Jan, and A.~J.
  Guttmann.
\newblock Self-avoiding walks on the simple cubic lattice.
\newblock {\em Journal of Physics A: Mathematical and General},
  33(34):5973--5983, 2000.

\bibitem{LatticeTrees}
Y.~M. Miranda and G.~Slade.
\newblock {The growth constants of lattice trees and lattice animals in high
  dimensions}.
\newblock {\em Electronic Communications in Probability}, 16:129--136, 2011.

\bibitem{LatticeTreesTerms}
Y.~M. Miranda and G.~Slade.
\newblock {Expansion in high dimension for the growth constants of lattice
  trees and lattice animals}.
\newblock {\em Combinatorics, Probability and Computing}, 22(4):527--565, 2013.

\bibitem{PG95}
P.~J. Peard and D.~S. Gaunt.
\newblock {$1/d$-expansions for the free energy of lattice animal models of a
  self-interacting branched polymer}.
\newblock {\em Journal of Physics A: Mathematical and General},
  28(21):6109--6124, 1995.

\bibitem{RanWelAni}
B.~M.~I. Rands and D.~J.~A. Welsh.
\newblock {Animals, trees and renewal sequences}.
\newblock {\em IMA Journal of Applied Mathematics}, 28(1):107--107, 1982.

\bibitem{WS90}
S.~G. Whittington and C.~E. Soteros.
\newblock Lattice animals: Rigorous results and wild guesses.
\newblock In {\em Disorder in Physical Systems}, pages 323--335. Clarendon
  Press, 1990.

\end{thebibliography}

\end{document}